\documentclass[11pt]{amsart}
\usepackage{amsfonts}
\usepackage{graphicx}
\usepackage{tabularx}
\usepackage{array}
\usepackage[usenames,dvipsnames]{color}
\usepackage{comment}
\usepackage{amsmath}
\usepackage{amsthm}
\usepackage{amssymb}
\usepackage{fullpage}
\usepackage[dvipsnames]{xcolor}
\usepackage{listings}

\newtheorem{theorem}{Theorem}[section]

\newtheorem{claim}[theorem]{Claim}
\newtheorem{lemma}[theorem]{Lemma}
\newtheorem{corollary}[theorem]{Corollary}
\theoremstyle{definition}

\def\epsilon{\varepsilon}

\title{Fractional colorings of partial $t$-trees with no large clique}

\author{Peter Bradshaw}
\address{Department of Mathematics, University of Illinois Urbana-Champaign}
\email{pb38@illinois.edu}
\thanks{This project was partially funded by NSF RTG grant DMS-1937241}

\begin{document}
\maketitle
\begin{abstract}
Dvo\v{r}\'ak and Kawarabayashi~\cite{DK} asked, what is the largest chromatic number attainable by a graph of treewidth $t$ with no $K_r$ subgraph?
In this paper, we consider the fractional version of this question.
We prove that if $G$ has treewidth $t$ and clique number $2 \leq \omega \leq t$, then $\chi_f(G) \leq t + \frac{\omega - 1}{t}$, and we show that this bound is tight for $\omega = t$. We also show that for each value $0 < c < \frac{1}{2}$, there exists a graph $G$ of a large treewidth $t$ and clique number $\omega = \lfloor (1 - c)t \rfloor$ satisfying $\chi_f(G) \geq t + 1 + \frac{1}{2}\log(1-2c) + o(1)$, which is approximately equal to the upper bound for small values $c$.
\end{abstract}

\section{Introduction}
\subsection{Background}
Determining the number of colors needed to 
properly color a graph belonging to a specific graph family is one of the oldest problems in graph theory.
For many graph families $\mathcal G$, the maximum chromatic number $\chi(G)$ achieved by a graph $G \in \mathcal G$ is achieved when $G$ is a clique. For example, the maximum chromatic number attainable by a planar graph is $4$, which is achieved by $K_4$. Similarly, the maximum chromatic number of a graph $G$ of maximum degree $\Delta$ is $\Delta + 1$, which is achieved by $K_{\Delta + 1}$. Other graph classes satisfying this property include graphs of bounded genus, bounded degeneracy, or bounded treewidth.

Given that the maximum chromatic number over all graphs in a family $\mathcal G$ is often attained by a clique, it is natural to ask about the maximum value $\chi(G)$ attained by a graph $G \in \mathcal G$ whose clique number $\omega(G)$ is bounded.
When $\mathcal G$ is the family of graphs of maximum degree $\Delta$, Brooks' theorem states that $\chi(G) \leq \Delta$ whenever $\omega(G) \leq {\Delta}$.
Moreover,
Borodin and Kostochka \cite{BK} conjectured that if $\Delta \geq 9$, then 
$\chi(G) \leq \Delta - 1$ for all $G \in \mathcal G$ satisfying $\omega(G) \leq \Delta - 1$, and Reed
\cite{ReedBK}
proved that this conjecture holds for $\Delta \geq \Delta_0$, where $\Delta_0 \leq  10^{14}$
is some large constant. Reed conjectured further that
$\chi(G) \leq \lceil \frac{\Delta + \omega(G) + 1}{2} \rceil$
for every graph $G \in \mathcal G$, and he proved in \cite{ReedConj} that there exists a universal value $\epsilon > 0$ such that $\chi(G) \leq \epsilon \omega(G) + (1 - \epsilon) (\Delta + 1)$. 
In the special case that $\omega(G) \leq r$ for a fixed constant $r$, Johansson \cite{Johansson} proved that $\chi(G) \leq 200 (r + 1) \frac{\Delta \log \log \Delta}{\log \Delta}$, and when $G$ is triangle-free, Molloy \cite{Molloy} proved that $\chi(G) \leq (1 + o(1)) \frac{\Delta}{\log \Delta}$. 

When we consider families of graphs with bounded genus, we see a similar pattern in which forbidding certain clique subgraphs reduces the maximum chromatic number attained by graphs in our family.
For instance, a planar graph $G$ may satisfy $\chi(G) = 4$, but Gr\"{o}tzsch's theorem 
\cite{Grotzsch}
states that $\chi(G) \leq 3$ whenever $G $ is triangle-free.
Similarly,
a graph $G$ of genus $g$ may satisfy $\chi(G) =  \Theta ( g^{1/2})$, but $\chi(G) = O \left ( \frac{g}{\log g} \right )^{1/3}$ whenever $G$ is triangle-free \cite{GT}.

In this paper, we focus on graphs of bounded treewidth, which are defined as follows. First, we define a \emph{$t$-tree} as a graph $G$ obtained by starting with a $t$-clique $K$ and then iteratively adding vertices $v$ for which $N(v)$ induces a $K_t$.
When $v$ is added to $G$, the $t$ neighbors of $v$ are called \emph{back-neighbors} of $v$. Note that by construction, each vertex of $G$ has exactly $t$ back-neighbors, except for the vertices of the initial clique $K$.
A \emph{partial $t$-tree} is a subgraph of a $t$-tree. Then, given a graph $G$, the \emph{treewidth} of $G$ is defined as the minimum value $t$ for which $G$ is a partial $t$-tree. For a more traditional definition of treewidth, see \cite{RS5}.
By ordering the vertices of a $t$-tree in the order that they are added to the graph, one sees that every partial $t$-tree $G$ is $t$-degenerate, and hence $\chi(G) \leq t + 1$. Furthermore, since $K_{t+1}$ is a $t$-tree, this upper bound is tight. 

Dvo\v{r}\'ak and Kawarabayashi \cite{DK} considered the problem of determining the largest value
$\chi(G)$
obtained by a partial $t$-tree $G$
when $\omega(G)$ is bounded. They found that when $\omega(G) \leq 2$, i.e.~when $G$ is triangle-free, it holds that $\chi(G) \leq \lceil \frac{t+3}{2} \rceil$. They also found that for each value $3 \leq \omega \leq t$, there exists a partial $t$-tree $G$ of clique number $\omega$ satisfying $\chi(G) > (1 - \frac{1}{2^{\omega-3}})t$,
and they left the task of determining upper bounds on $\chi(G)$ for these larger values of $\omega$ as an open problem.

\subsection{Fractional chromatic number}

The fractional chromatic number of a graph is defined as follows \cite{FracDef}. Given a graph $G$, a \emph{fractional coloring} is an assignment of a set $\phi(v) \subseteq \mathbb R$ of Lebesgue measure $\mu(\phi(v)) = 1$ to each vertex $v \in V(G)$ so that $\phi(u) \cap \phi(v) = \emptyset$ for each edge $uv \in E(G)$. Then, the \emph{fractional chromatic number} of $G$, written $\chi_f(G)$, is the infimum of $\mu \left ( \bigcup_{v \in V(G)} \phi(v) \right )$ over all fractional colorings $\phi$ of $G$.
Johnson and Rodger \cite{FracDef} point out that the infinum in this definition is in fact a minimum; that is, a graph $G$ satisfying $\chi_f(G) = k$ has a fractional coloring $\phi:V(G) \rightarrow 2^{\mathbb R}$ satisfying  $\mu \left ( \bigcup_{v \in V(G)} \phi(v) \right ) = k$.
Given a fractional coloring $\phi:V(G) \rightarrow 2^{\mathbb R}$, 
if $\bigcup_{v \in V(G)} \phi(v)  \subseteq \mathcal C$, then we say that 
$G$ is colored using the color set $\mathcal C$.

Equivalently, for an integer $b \geq 1$, a \emph{proper $b$-coloring} of $G$ is an assignment $\phi: V(G) \rightarrow \binom{\mathbb N}{b} $ of $b$ distinct colors to each vertex $v \in V(G)$ such that $\phi(u) \cap \phi(v) = \emptyset$ for each edge $uv \in E(G)$. If $G$ has a proper $b$-coloring $\phi:V(G) \rightarrow \binom{[k]}{b}$ using the color set $[k] = \{1, \dots, k\}$, then 
we say that $G$ is $(k,b)$-colorable.
We define $\chi_b(G)$ to be the minimum value $k$ so that $G$ is $(k,b)$-colorable; then, $\chi_f(G) = \inf_{b \rightarrow \infty} \frac{\chi_b(G)}{b}$.

Since a proper graph coloring is a proper $b$-coloring with $b = 1$, it follows that $\chi_f(G) \leq \chi(G)$ for every graph $G$.
A graph's chromatic number and fractional chromatic number may be equal; for example, $\chi(K_r) = \chi_f(K_r) = r$ for each $r \geq 1$. On the other hand, 
the difference between the chromatic number and the fractional chromatic number of a single graph may be arbitrarily large. As an example, we consider the \emph{Kneser graph} $K(n,k)$, whose vertex set is the collection of subsets of $\{1, \dots, n\}$ of size $k$, and whose edge set consists of the pairs of vertices whose corresponding subsets are disjoint.
When $G = K(n,k)$,
the chromatic number of $G$ is
$\chi(G) = n - 2k + 2$ \cite{Kneser},
while the fractional chromatic number of $G$ is 
$\chi_f(G) = n/k$.

Similarly to the chromatic number, it is natural to ask about the maximum value $\chi_f(G)$ attained by a graph $G$ in some family $\mathcal G$ when the clique number $\omega(G)$ is bounded.
In this vein, Harris \cite{Harris} conjectured that $\chi_f(G) = O \left ( \frac{d}{\log d} \right )$ for every $d$-degenerate triangle-free graph $G$, and this intriguing conjecture is still open.
    Additionally, one may ask for a fractional version of Borodin and Kostochka's conjecture---that is, whether $\chi_f(G) \leq \Delta - 1$ for all graphs $G$ of maximum degree $9 \leq \Delta < \Delta_0$ with no $K_{\Delta}$ subgraph, where $\Delta_0$ is the value from Reed's result \cite{ReedBK}. Partial progress has been been made toward answering this question; King, Lu, and Peng \cite{King} showed that for all such graphs $G$, $\chi_f(G) \leq \Delta- \frac{2}{67}$, and Hu and Peng \cite{HP} recently announced a better bound of $\chi_f(G) \leq \Delta - \frac{1}{8}$. In fact, these upper bounds hold even for all $\Delta \geq 4$, except when $G$ is the squared cycle $C_8^2$ or the strong product $C_5 \boxtimes K_2$.

\subsection{Our results}
In this paper, similarly to Dvo\v{r}\'ak and Kawarabayashi \cite{DK}, we consider the question of determining the largest value $\chi_f(G)$ obtained by a partial $t$-tree $G$ when the clique number of $G$ is at most some fixed value $\omega$. 
As in Borodin and Kostochka's conjecture, we mainly consider cases where $\omega$ is close to the maximum possible clique number, which in our case of partial $t$-trees is $t+1$. We find that whenever $\omega \leq t$, the maximum value $\chi_f(G)$ attainable by a partial $t$-tree $G$ with clique number $\omega$ is strictly less than $t+1$. Specifically, we prove the following theorem.

\begin{theorem}
\label{thm:intro_UB}
Let $G$ be a graph of treewidth $t \geq 1$, and let $2 \leq \omega \leq t$. If the largest clique in $G$ has at most $\omega$ vertices, then $\chi_f(G) \leq t + \frac{\omega - 1}{t}$.
\end{theorem}

Theorem \ref{thm:intro_UB} tells us that when $G$ has treewidth $t$ and clique number $\omega$, it holds that $\chi_f(K_{t+1}) - \chi_f(G) \geq \frac{t - \omega + 1}{t}$; in other words, the fractional chromatic number of $G$ is at least $\frac{t - \omega + 1}{t}$ below the maximum attainable fractional chromatic number of a partial $t$-tree.
When $\omega = t$, Theorem \ref{thm:intro_UB} gives an upper bound of $\chi_f(G) = t + 1 - \frac{1}{t}$, and we will see that this upper bound is tight.
Furthermore, when $\omega = \lfloor  (1-c)t \rfloor $ for a constant $0 < c < \frac{1}{2}$, Theorem \ref{thm:intro_UB} gives an upper bound of $\chi_f(G) < t + 1 - c $, 
implying that $\chi_f(K_{t+1}) - \chi_f(G) > c $.
The following theorem tells us that when $c$ is small, 
this gap of size $c$ is close to best possible.

\begin{theorem}
\label{thm:intro_LB}
Let $0 < c < \frac{1}{2}$, and let $t$ be a large integer. There exists a graph $G$ of treewidth $t$ and clique number $\lfloor (1 - c)t \rfloor$ satisfying 
\[ \chi_f(G) \geq t + 1 + \frac{1}{2}\log (1-2c) - o(1).\]
\end{theorem}
When $c$ is small, the approximation $ \log(1-2c) \approx -2c -2  c^2$ tells us that this lower bound is roughly equal to $t + 1 - c - c^2$, meaning that the gap 
$\chi_f(K_{t+1}) - \chi_f(G)$ cannot be much larger than $c + c^2$.

The paper is organized as follows. In Section \ref{sec:game}, we introduce an online coloring game
between Alice and Bob, in which Alice wishes to fractionally color a graph with a color set of small measure, and Bob wishes to create a graph that requires a color set of large measure.
This online coloring game 
gives us a more convenient setting for proving Theorems \ref{thm:intro_UB} and \ref{thm:intro_LB}. Our online coloring game uses the same fundamental ideas as the online coloring game introduced by Dvo\v{r}\'ak and Kawarabayashi \cite{DK}.
In Section \ref{sec:UB}, we describe a strategy for Alice in our online coloring game that succeeds with a color set of bounded measure, and we show that this strategy implies Theorem \ref{thm:intro_UB}. In Section \ref{sec:LB}, we describe a strategy of Bob in our game that forces Alice to use a color set of large measure, and we show that this strategy implies Theorem \ref{thm:intro_LB}. 

\section{An online coloring game}
\label{sec:game}
In this section, we describe an online coloring game which gives a convenient setting for considering fractional colorings of partial $t$-trees. The game we describe is fundamentally a fractional equivalent of the game introduced by Dvo\v{r}\'ak and Kawarabayashi \cite{DK} to describe proper colorings of partial $t$-trees. However, our formal approach is rather different, so we 
give a full exposition of our game for the sake of completeness.

Given integers $t \geq 1$ and $2 \leq \omega \leq t$, we define the online coloring game for $K_{\omega+1}$-free partial $t$-trees. The game is played by two players, Alice and Bob.
Throughout the game, a graph with red and blue edges is constructed.
As an initial step of the game (which we call Turn $0$), Bob chooses a positive integer $N$, which determines how many turns the game will last. 
Next, 
Alice constructs a graph $G$ consisting of a single $K_t$ with red edges, and she assigns each of the $t$ vertices $u \in V(G)$ a measure-$1$ color set $\phi(u) \subseteq \mathbb R$.
Then, Alice and Bob begin taking turns.
On each turn of the game, Bob first adds a new uncolored vertex $v$ to $G$, so that $G[N(v)]$ is a clique of size $t$. Bob also colors the edges incident to $v$ with the colors red and blue. 
Then, Alice gives $v$ a color set $\phi(v) \subseteq \mathbb R$ of measure $1$ to this new vertex $v$.
Alice must obey the rule that any two vertices joined by a blue edge have disjoint color sets, and Bob must obey the rule that each $K_{\omega + 1}$ subgraph of $G$ has at least one red edge. 
After $N$ turns, the game ends.
During the game, Alice tries to minimize the value $\mu \left (\bigcup_{v \in V(G)} \phi(v)  \right )$, which is the measure of the overall color set that she has used during the game, and Bob tries to maximize this value.

If $G$ is a graph whose edges are colored red and blue, then for each vertex $v \in V(G)$, we write $N_R(v)$ ($N_B(v)$) for the set of vertices in $G$ joined to $v$ by a red (blue) edge. We say that vertices in $N_R(v)$ are \emph{red neighbors} of $v$, and we define \emph{blue neighbors} similarly.


The following lemma shows a certain equivalence between the fractional coloring problem and the online coloring game described above for partial $t$-trees of fixed clique number. The lemma uses the same ideas as \cite[Lemma 5]{DK}.

\begin{lemma}
\label{lem:game}
Let $k,t \geq 1$ and $2 \leq \omega \leq t$ be fixed integers. Every graph $G$ of treewidth $t$ and clique number at most $\omega$ satisfies $\chi_f(G) \leq k$ if and only if Alice has a strategy to complete the online coloring game for $K_{\omega+1}$-free partial $t$-trees using some color set $\mathcal C$ of measure $k$.
\end{lemma}
\begin{proof}
First, suppose that $\chi_f(G) \leq k$ for every graph $G$ of treewidth $t$ and clique number at most $\omega$. 
We show that Alice has a strategy to complete the online coloring game for $K_{\omega+1}$-free partial $t$-trees using some color set $\mathcal C$ of measure $k$.
Let $N$ be the number of turns selected by Bob.
We define a $t$-tree $H$ with a red-blue edge-coloring as follows. We let $H$ begin with a single $K_t$ with red edges, which we call $K'$. For notational purposes, we write $V(K') = S_{0,K'}$.

Then, for $i = 1, \dots, N$ and for each $t$-clique $K$ in $H$, we iterate the following: 
\begin{quote}
Add a set $S_{i,K}$ of $2^{t}$ new vertices to $H$. For each $v \in S_{i,K}$, let $N(v) = K$. 
Then, edge-color the sets $\{e:v \in e\}$ for $v \in S_{i,K}$ with each of the $2^{t}$ possible red-blue edge-colorings---that is, so $N_R(v) \neq N_R(w)$ holds for each distinct pair $v,w\in S_{i,K}$. Finally, if some $v \in S_{i,K}$ belongs to a blue $K_{\omega + 1}$ in $H$, then delete $v$ from $H$.
\end{quote}
By definition, the graph $H$ is a $t$-tree, and hence the spanning subgraph $H'$ induced by the blue edges of $H$ is a $K_{\omega + 1}$-free partial $t$-tree. 
Hence, there exists a color set $\mathcal C$ of measure $k$ for which there exists a fractional coloring $\psi:V(H) \rightarrow 2^{\mathcal C}$ of $H'$.

We write $G$ for the graph that Bob creates during the game, and we use $\phi$ to denote the coloring function that Alice creates during the game.
Now, we show that Alice has a strategy to complete the game so that $\phi(v) \subseteq \mathcal C$ for each vertex $v$ that appears in $G$ during the game. 
To show this, we prove the stronger statement that Alice has a strategy such that after each Turn $i$, the graph $G$ constructed in the online coloring game appears as a vertex and edge-colored subgraph $H^* \subseteq H$, where each vertex of $H^*$ belongs to a set $S_{j,K}$ for which $j \leq i$. This statement implies that after the end of Turn $N$, $G$ is colored using the color set $\psi(V(H)) \subseteq  \mathcal C$, which is exactly what we need to prove.

We prove this stronger statement by induction on $i$. 
Before Alice and Bob begin taking turns (i.e.~on Turn 0), Alice colors the initial $K_t$ of $G$ as $\psi(S_{0,K'})$, which completes the base case. Now, suppose that $i \geq 1$ and the statement holds for values up to $i-1$. At the beginning of Turn $i$, 
by the induction hypothesis,
$G$ is isomorphic as a vertex and edge-colored subgraph $H^* \subseteq H$ consisting of vertices belonging to sets $S_{j,K}$ for $j \leq i-1$. Now, on Turn $i$, Bob chooses a clique $K$ in $G$ and adds a new vertex $v$ so that $N(v) = K$, and Bob gives a red-blue edge-coloring to the edges incident to $v$. Then, Alice finds the clique $K^*$ in $H^*$ corresponding to $K$ and locates a vertex $v^* \in S_{i,K^*}$ so that
the edge colors incident to $v^*$ correspond with the edge colors that Bob has assigned to the edges incident to $v$. 
We know that $v^*$ was not deleted while constructing $H$, since Bob is not allowed to create a blue $K_{\omega + 1}$ with his edge-coloring.
Then, Alice sets $\phi(v) = \psi(v^*)$. After this, we observe that $G$ is isomorphic to $H^* \cup \{v^*\} \subseteq H$ as a vertex and edge-colored graph, and each vertex of $H$ belongs to a set $S_{j,K}$ for which $j \leq i$. Hence, induction is complete, and Alice completes the game using the color set $\mathcal C$.

On the other hand, suppose Alice has a strategy to complete the online coloring game for $K_{\omega+1}$-free partial $t$-trees using some color set $\mathcal C$ of measure $k$.
We show that every $K_{\omega+1}$-free partial $t$-tree has a fractional coloring using the color set $\mathcal C$.
To this end, let $G$ be a $K_{\omega+1}$-free partial $t$-tree. By definition, there exists a $t$-tree $H$ for which $G$ is a spanning subgraph. We give $H$ a red-blue edge-coloring so that the blue subgraph of $H$ is isomorphic to $G$. Furthermore, since $H$ is a $t$-tree, $H$ may be constructed by starting with a $K_t$ with vertices $v_1, \dots, v_t$ and then iteratively adding vertices whose neighborhoods induce cliques of size $t$ in the existing graph, giving an ordering $v_1, \dots, v_n$ for $V(H)$.
For technical reasons, we add a set $U = \{u_1, \dots, u_t\}$ of $t$ more vertices to $H$, and we add a red edge between each pair of vertices in $U$. Furthermore, for each vertex $v_i \in \{v_1, \dots, v_t\}$, we add a red edge from $v_i$ to $u_i, \dots, u_t$. Then, we order $V(H)$ as $(u_1, \dots, u_t, v_1, \dots, v_n)$. We observe that for each vertex $v_i \in V(H)$, the back-neighbors of $v_i$ form a $t$-clique in $H$, so $H$ is still a $t$-tree.

Now, we give a fractional coloring $\phi:V(H) \rightarrow 2^{\mathcal C}$ to the vertices of $H$ 
by using
Alice's strategy in the online coloring game, as follows. We assume that Bob chooses to play for $N= n$ turns.
First, we
let Alice choose arbitrary measure-$1$ subsets $\phi(u_i) \subseteq \mathcal C$ for each 
$i \in \{1, \dots, t\}$.
 Then, we consider the vertices $v_1, \dots, v_n$ one at a time.
 When each vertex 
 $v_i$ is considered, we interpret this vertex as a move from Bob in the online game, and
we let Alice assign a measure-$1$ subset $\phi(v_i) \subseteq \mathcal C$.
We observe that as $H$ has no blue $K_{\omega+1}$, each vertex that we consider corresponds with a possible move from Bob in the online coloring game for $K_{\omega+1}$-free $t$-trees.
Hence, by following Alice's strategy, we fractionally color $H$ with a function $\phi$ so that $\phi(u) \subseteq \mathcal C$ for each $u \in V(H)$, and so that $\phi(u) \cap \phi(v) = \emptyset$ for each vertex pair $u,v \in V(H)$ joined by a blue edge. 
Then, since the blue subgraph of $H$ is isomorphic to $G$, we find a fractional coloring of $G$ using the color set $\mathcal C$.
Hence, $\chi_f(G) \leq k$, and the proof is complete.
\end{proof}

\section{An upper bound}
\label{sec:UB}
In this section, we show that a partial $t$-tree $G$ with clique number at most $2 \leq \omega \leq t$ satisfies $\chi_f(G) \leq t + \frac{\omega - 1}{t}$, proving Theorem \ref{thm:intro_UB}. In order to prove this upper bound, we define
a strategy for Alice in the online coloring game for $K_{\omega+1}$-free partial $t$-trees, and then we apply Lemma \ref{lem:game}.
Before proving our upper bound, we need the following lemma.
\begin{lemma}
\label{lem:back}
Let $K$ be a clique on $t$ vertices with a red-blue edge coloring. Suppose that each vertex $v \in V(K)$ is assigned a set $\gamma(v)$ of exactly $t$ colors, subject to the following:
\begin{itemize}
\item If $u, v \in V(K)$ are joined by a blue edge, then $\gamma(u) \cap \gamma(v) = \emptyset$, and
\item If $u,v \in V(K)$ are joined by a red edge, then $|\gamma(u) \cap \gamma(v)| = 1$.
\end{itemize}
If the largest blue clique 
in $K$ is of order $\omega$ for some $0 \leq \omega \leq t$, 
then
$\left | \bigcup_{v \in V(K)} \gamma(v) \right | \leq t^2 - t + \omega$.
\end{lemma}
\begin{proof}
We order the vertices of $K$ as $v_1, \dots, v_t$ so that $v_1, \dots, v_{\omega}$ induce a maximum blue clique. For each value $j \in \{1, \dots, t\}$, we let $S_j = \bigcup_{i = 1}^{j} \gamma(v_i)$. Observe that $|S_{\omega}| = t \omega$. Furthermore, for each value $i > \omega$, $v_i$ is joined to at least one neighbor in $v_1, \dots, v_{\omega}$ by a red edge, as otherwise $v_1, \dots, v_{\omega}, v_i$ would form a blue clique of size $\omega  +1 $, a contradiction. 
Therefore, for each $i > \omega $, $\gamma(v_i)$ contains at least one color from $S_{i-1}$, and hence $|S_i| \leq |S_{i-1}| + (t-1)$. Hence, $|S_t| \leq  \omega t + (t - \omega) (t-1) = t^2 - t + \omega$, completing the proof.
\end{proof}

Now, we are ready to prove Theorem \ref{thm:intro_UB}, which states that $\chi_f(G) \leq t + \frac{\omega - 1}{t}$ for every $K_{\omega+1}$-free partial $t$-tree $G$. 

\begin{proof}[Proof of Theorem \ref{thm:intro_UB}]
By Lemma \ref{lem:game}, it suffices to show that Alice has a strategy in the online coloring game for $K_{\omega+1}$-free partial $t$-trees by which each vertex $v$
receives a measure-$1$ set $\phi(v)\subseteq (0, t + \frac{\omega-1}{t})$.
We let Alice give each vertex $v$ a set $\phi(v)$ equal to the union of $t$ intervals from the discrete set $\{ (\frac{i-1}{t}, \frac{i}{t} ): 1 \leq i \leq t^2 + \omega - 1\}$.
For ease of presentation, 
we say that Alice assigns each vertex $v$ a set $\gamma(v)$ of exactly $t$ colors from the set $\{1, \dots, t^2 + \omega - 1\}$. Then, we will let $\phi(v) = \bigcup_{i \in\gamma(v)}(\frac{i-1}{t}, \frac{i}{t} ) $.

We prove the stronger statement that Alice has a strategy to assign color sets $\gamma(v)$ of size $t$ to vertices $v$ in the graph $G$ created throughout the game while satisfying the additional condition
that
$\gamma(v) \cap \gamma(w) = \emptyset$ for any two vertices $v,w \in V(G)$ joined by a blue edge, and so that $|\gamma(v) \cap \gamma(w)| = 1$ for any two vertices $v, w \in V(G)$ joined by a red edge.


When the game begins, 
$G$ consists of vertices $v_1, \dots, v_t$ forming a red clique.
Alice gives each vertex $v_i$ a set $\gamma(v_i)$ so that $\gamma(v_i) \cap \gamma(v_j) = \{1\}$ holds for each distinct vertex pair $v_i, v_j$. This is possible, as Alice needs only $t (t-1) + 1 < t^2 + \omega - 1$ colors for this step. After this initial phase of the game, Alice's coloring satisfies our condition.

Now, suppose that Bob has added a vertex $v$ to $G$ so that $N(v)$ induces a $t$-clique $K$ in $G$. We observe that for each red neighbor $u$ of $v$,
$u$ has $t-1$ neighbors $w \in V(K),$ and $\gamma(v)$ shares at most one color with each set $\gamma(w)$; thus, $\gamma(u)$ contains some color $c_u$ that does not appear in $\gamma(V(K) \setminus \{u\})$. 
For each red neighbor $u$ of $v$, Alice adds such a color $c_u$ to $\gamma(v)$.

Next, we claim that Alice can complete $\gamma(v)$ to a set of $t$ colors from $\{1, \dots, t^2 + \omega - 1\}$ while satisfying our condition. We let $\ell = |N_B(v)|$. Since Alice's previous step puts $t - \ell$ colors in $\gamma(v)$, Alice still needs to give $\ell$ additional colors to $\gamma(v)$. 
We say that the colors of $\bigcup_{u \in N_R(v)} \gamma(u)$ are \emph{forbidden}, as these colors cannot further be added to $\gamma(v)$. 
There are at most $ t ( t - \ell)$ such forbidden colors.
Next, we consider two cases.
\begin{enumerate}
\item If $\ell \leq \omega - 1$, then the $\ell$ blue neighbors of $v$ have color sets altogether containing at most $t \ell $ colors, and we say that these colors are forbidden.
In total, at most $t (t - \ell) + t \ell = t^2$ colors are forbidden. 
Since any color that is not forbidden can still be added to $\gamma(v)$, Alice has at least $\omega - 1 \geq \ell$ legal colors available with which to complete the set $\gamma(v)$.
\item If $\ell \geq \omega $, then we consider the clique $K'$ induced by the $\ell$ blue neighbors of $v$. Since the largest blue clique in $G$ has $\omega$ vertices, $K'$ cannot contain a blue $K_{\omega}$-subgraph, and hence the largest blue clique in $K'$ is of order at most $\omega - 1$. By Lemma \ref{lem:back}, the set $\gamma(V(K'))$ contains at most $\ell^2  - \ell + \omega - 1$
colors, and we say that these colors are forbidden. Therefore, considering the at most $t (t - \ell)$ colors 
which are already forbidden, altogether at most $t(t - \ell) + \ell^2 - \ell + \omega - 1$ colors are forbidden. However, it holds that 
\[ (t^2 + \omega - 1) - (t(t - \ell) + \ell^2 - \ell + \omega - 1) \geq \ell,\]
since simplification shows that this inequality is equivalent to $\ell(t - \ell) \geq 0$.
Therefore, since any color that is not forbidden can still be added to $\gamma(v)$, Alice has at least $\ell$ legal colors with which to complete the set $\gamma(v)$.
\end{enumerate}
Therefore, for any move that Bob makes, Alice has a legal response that maintains the condition that $\gamma(v) \in \binom{ [ t^2 - \omega + 1]}{t}$ for each $v \in V(G)$ and $|\gamma(v) \cap \gamma(w)| = 1$ for each red edge $vw \in E(G)$. Therefore,
letting $\phi(v) = \bigcup_{i \in\gamma(v)}  (\frac{i-1}{t}, \frac{i}{t} ) $
for each $v \in V(G)$,
Alice can complete the online coloring game for $K_{\omega + 1}$-free partial $t$-trees with the color set $(0,t+\frac{\omega - 1}{t})$. Then, by Lemma \ref{lem:game}, 
every partial $t$-tree of clique number 
$\omega$ has fractional chromatic number at most $t + \frac{\omega - 1}{t}$, completing the proof.
\end{proof}

\section{Lower bounds}
\label{sec:LB}
In this section, we construct partial $t$-trees with clique number $ (\frac{1}{2} + \epsilon)t < \omega \leq t$ and with large fractional chromatic number. Our constructions imply Theorem \ref{thm:intro_LB}. Rather than directly constructing these partial $t$-trees, we describe strategies by which Bob can force Alice to use a color set of large measure in the online coloring game. Then, the partial $t$-trees attaining large fractional chromatic numbers can be obtained from the method of Lemma \ref{lem:game}.

When we describe Bob's strategies in the online coloring game, it will often be convenient to allow Bob to
add vertices $v$ to the constructed graph $G$ so that $N(v)$ forms a clique with fewer than $t$ vertices. 
The following lemma shows that the online coloring game does not change even if we give Bob this extra freedom. 

\begin{lemma}
\label{lem:small_clique}
Let $t \geq 1$ be an integer.
If $H$ is a $t$-tree, then every clique $K$ in $H$ of size at most $t$ is a subgraph of a $t$-clique $K'$.
\end{lemma}
\begin{proof}
We induct on the number of vertices in $H$. If $H$ has $t$ vertices, then $H \cong K_t$, so the claim clearly holds. Now, suppose that the claim holds for a $t$-tree $H$. We show that if we add a vertex $v$ to $H$ so that $N(v)$ induces a $t$-clique, then the claim holds for $H \cup \{v\}$. Consider a clique $K \subseteq H \cup \{v\}$. If $K \subseteq H$, then the claim holds by the induction hypothesis. Otherwise, $v \in K$, and hence $V(K) \subseteq N[v]$. Then, $K$ is a subgraph of the $(t+1)$-clique with vertex set $N[v]$, and thus $K$ can be extended to a $t$-clique $K'$ obtained by deleting some vertex from $N[v]$.
\end{proof}
Hence, by Lemma \ref{lem:small_clique}, if Bob wishes to add a vertex $v$ to $G$ so that $N(v)$ induces a clique $K$ on fewer than $t$ vertices, Bob may first add $v$ in this way, and then he may
choose a clique $K' \supseteq K$ on $t$ vertices and add a red edge between $v$ and each vertex of $K' \setminus K$. These extra red edges do not affect Alice's legal moves, and they can only help Bob later in the game. Therefore,
when we describe Bob's strategies in this section, we often let Bob add vertices $v$ to $G$ so that $|N(v)| < t$, and we will tacitly assume that Bob adds extra red edges incident to $v$ so as to make his move legal within the game's formal rules.

For integers $t \geq 1$ and $1 \leq \omega \leq t+1$, we define the function $f(t,\omega) = \sup \{\chi_f(G): G \in \mathcal G_{t,\omega}\}$, where $\mathcal G_{t, \omega}$ is the set of all graphs of treewidth $t$ and clique number at most $\omega$.
Theorem \ref{thm:intro_UB} shows that $f(t,\omega) \leq t + \frac{\omega-1}{t}$. We will be interested in finding constructions that give lower bounds for $f(t, \omega)$.

In the following theorem, we describe a recursive construction for a partial $t$-tree which gives a recursive lower bound for $f(t, \omega)$.

\begin{theorem}
\label{thm:recursion}
If $\frac{t+1}{2} \leq \omega \leq t$, then
\[f(t,\omega ) \geq t + 1 - \sum_{i = 1}^{t - \omega + 1 } \frac{1}{f(t-2i+1,\omega - i + 1)}. \]
\end{theorem}
\begin{proof}
We choose $\epsilon > 0$ to be an arbitrarily small value.
We consider the online coloring game for $K_{\omega + 1}$-free partial $t$-trees, and we write $\phi$ for Alice's coloring function.
We show that
Bob can force Alice to use a color set of measure at least 
\[t + 1 - \sum_{i = 1}^{t-\omega+1} \frac{1}{f(t-2i+1,\omega - i + 1) - \epsilon} . \] 

Bob iterates through $i =  1, \dots, t-\omega+1$, completing iteration $i$ as follows.
At the beginning of iteration $i$,
we assume that
Bob has already added
vertices forming a set $U = \{u_1, \dots, u_{i-1}, v_1, \dots, v_{i-1}\}$. 
We also assume that Bob has colored every edge $u_j v_j$ red for $j \in \{1, \dots, i-1\}$, and we assume that Bob has colored every other edge of $G[U]$ blue.
Bob begins iteration $i$ by adding a vertex $v_i$ and letting every vertex of $U$ be a blue neighbor of $v_i$. Next, Bob constructs a graph $H_i$ of treewidth $t-2i + 1$ and clique number at most $\omega - i + 1$ with fractional chromatic number at least $f(t-2i+1, \omega - i+1) - \epsilon$,
in which each vertex of $H_i$ has $v_i$ as a red back-neighbor and every vertex of $U$ as a blue back-neighbor.
Since the largest blue clique induced by $U$ 
has size $i-1$, and since the largest blue clique in $H_i$ contains at most $\omega - i + 1$ vertices, Bob does not create a blue $K_{\omega + 1}$. Furthermore, each time Bob adds a new vertex $u$ to $H_i$, $u$ has at most $t-2i+1$ back-neighbors in $H_i$, as well as $2i-1$ back-neighbors $U \cup \{ v_i \}$, so $u$ altogether has at most $t$ back-neighbors.
Since each vertex of $U$ is adjacent to each vertex of $H_i$, the back-neighbors of $u$ form a clique.
Therefore, each of Bob's moves is legal. 
We make the following claim.

\begin{claim}
$H_i$ contains some vertex $u_i$ satisfying $\mu ( \phi(u_i) \cap \phi(v_i) ) \leq \frac{1}{f(t-2i+1,\omega - i  + 1) - \epsilon } $. 
\end{claim}
\begin{proof}
Suppose that the claim does not hold. Then, since $|V(H_i)|$ is finite, it holds that for a sufficiently small value $\alpha > 0$,
each vertex of $H_i$ shares at least $\frac{1}{f(t-2i+1,\omega - i + 1) - \epsilon - \alpha}$ colors with $\phi(v_i)$. We define, for each vertex $w \in V(H_i)$, the set $S_w = \phi(w) \cap \phi(v_i)$ of measure at least $\frac{1}{f(t-2i+1, \omega - i + 1) -  \epsilon - \alpha}$, so that $S_w \cap S_{w'} = \emptyset$ for each pair of neighbors $w$ and $w'$ and so that $\mu \left (\bigcup_{w \in H_i} S_w \right ) \leq 1$. Then, for each $w \in V(H_i)$, we assign the color set 
\[\psi(w)  =  \left \{({f(t-2i+1,\omega - i + 1) - \epsilon - \alpha} ) x : x \in S_w \right \}.\] 
We note that $\mu (\psi(w)) \geq 1$ for each $w \in V(H_i)$, and $\psi(w) \cap \psi(w') = \emptyset$ for each edge $w w' \in E(H_i)$.
We note further that since $\phi(v_i)$ has measure $1$,
 $\mu \left ( \bigcup_{w \in H_i} \psi(w) \right ) \leq f(t-2i+1,\omega - i+1) -  \epsilon - \alpha $.
This implies that $\chi_f(H_i) < f(t-2i+1,\omega - i+ 1) - \epsilon $, which is a contradiction. Thus, the claim holds.
\end{proof}
By the claim, Bob chooses some vertex $u_i \in V(H_i)$ for which $|\phi(u_i) \cap \phi(v_i)| \leq \frac{1}{f(t-2i+1,\omega - i + 1) - \epsilon }$. 
This completes iteration $i$. We observe that by construction, all edges induced by the set $\{u_1, \dots, u_i, v_1, \dots, v_i\}$ are blue, except for the edges $u_j v_j$ for $1 \leq j \leq i$, so the condition that we assumed at the beginning of iteration $i$ is still valid for the next iteration.

After completing all $t - \omega + 1$ iterations, Bob adds vertices vertices $w_1, \dots, w_{2 \omega - t - 1}$ so that $w_i$ has back-neighbors \[v_1, \dots, v_{t - \omega + 1}, u_1, \dots, u_{t - \omega + 1}, w_1, \dots w_{i-1}.\]
We observe that the vertices $u_1, \dots, u_{t - \omega + 1}, v_1, \dots, v_{t - \omega + 1}, w_1, \dots, w_{2 \omega - t - 1}$ form a clique $K'$ in $G$ with a red perfect matching of size $t-\omega + 1$ and all other edges colored blue. Therefore, no three vertices of $K'$ have a nonempty intersection of their color sets, and hence by the inclusion-exclusion principle,
 \[ \mu \left (\phi(V(K')) \right ) \geq t + 1 - \sum_{i = 1}^{t-\omega+1} \frac{1}{f(t-2i+1,\omega - i+1) - \epsilon }  .\]
 Hence, by Lemma \ref{lem:game}, 
 $f(t,\omega) \geq t + 1 - \sum_{i = 1}^{t-\omega+1} \frac{1}{f(t-2i+1,\omega - i + 1)- \epsilon }  $. 
  Letting $\epsilon$ tend to $0$ completes the proof.
\end{proof}

While Theorem \ref{thm:recursion} appears slightly unwieldy, it has a number of corollaries that are clearer to state.
First, we can use the trivial lower bound $f(t, \omega) \geq \omega$
and the construction in Theorem \ref{thm:recursion}
to obtain the following corollary.
\begin{corollary}
\label{cor:fixed_r}
For $t \geq 1$ and $ \frac{1}{2}(t+1) \leq \omega \leq t$, 
there exists a graph $G$ of treewidth $t$ and clique number $\omega$
satisfying 
\[\chi_f(G) \geq t + 1 - \sum_{i = 1}^{t - \omega + 1 } \frac{1}{\omega - i + 1}.\]
\end{corollary}
Corollary \ref{cor:fixed_r} shows that 
the upper bound of $t + 1 - \frac{1}{t}$ for the fractional chromatic number of graphs of treewidth $t$ and clique number $t$ is tight. 
For a fixed value $r \geq 0$ and large $t$, Corollary \ref{cor:fixed_r} also tells us that there exists a graph $G$ of treewidth $t$ and clique number $t - r$ satisfying
\[\chi_f(G) \geq t + 1 - \frac{(1 + o(1))(r+1)}{t}.\]
Furthermore, using a harmonic sum, Corollary \ref{cor:fixed_r} shows that for all values $t \geq 1$ and $ \frac{1}{2} (t+ 1) \leq \omega \leq t$,
\begin{equation}
\label{eqn:t}
f(t, \omega) >  t - \log t = (1 + o(1))t,
\end{equation}
which gives us the asymptotic growth rate of $f(t,\omega)$ for all  $\omega \geq \frac{1}{2}(t + 1)$. Using this fact, we can obtain the following tighter corollary, which implies Theorem \ref{thm:intro_LB},
\begin{corollary}
\label{cor:asy}
Let $0 < c < \frac{1}{2}$, and let $t$ be a large integer. If $\omega = \lfloor (1 - c)t \rfloor$, then  
\[ f(t, \omega)  \geq t + 1 + \frac{1}{2}\log (1-2c) - o(1).\]
\end{corollary}
We compare this lower bound with the upper bound $f(t, \omega) \leq t + 1 - c$ given by Theorem \ref{thm:intro_UB}.
Using the approximation $\log(1-2c) \approx -2c - 2c^2$, we see that when 
$c$ is small, this lower bound is within roughly $c^2$ of being best possible.
\begin{proof}[Proof of Corollary \ref{cor:asy}]
We set $\omega = \lfloor (1-c)t \rfloor$ and prove that 
$f(t, \omega) \geq 
t + 1 + \frac{1}{2}\log (1-2c) - o(1).
$
We begin with the recursive bound from Theorem \ref{thm:recursion}. 
\[f(t,\omega ) \geq t + 1 - \sum_{i = 1}^{t - \omega + 1 }  \frac{1}{f(t-2i+1,\omega - i + 1
)}. \]
Since $\omega > \frac{1}{2}t$,
it follows that $\omega - i + 1 > \frac{1}{2}(t - 2i + 2) $. Furthermore, for each value $i$ in our sum, $t - 2i + 1 = \Omega(t)$. Therefore, it follows from Equation (\ref{eqn:t}) that for each value $i$ in the sum, 
$f(t - 2i + 1, \omega - i) = (1 + o(1)) (t - 2i+1)$. Therefore, we find that 
\[f(t, \omega) \geq t + 1 - (1+o(1))\sum_{i = 1}^{t - \omega + 1} \frac{1}{t - 2i + 1}  \geq t + 1 - (1 + o(1)) \sum_{x = \lfloor \frac{2\omega - t - 1}{2} \rfloor }^{\lfloor \frac{t-1}{2} \rfloor } \frac{1}{2x}.\]
For a strictly decreasing positive function $g(x)$, it holds that 
$\sum_{x = a}^b 
g(x) < g(a) + \int_a^b g(x) dx $,
so we approximate the sum in our bound with an integral and find that
\begin{eqnarray*}
f(t,\omega) &>& t + 1 - (1+o(1)) \int_{\lfloor \frac{2\omega - t - 1}{2} \rfloor}^{\lfloor \frac{t-1}{2} \rfloor} \frac{1}{2x} dx \\
&=& t + 1 + \left (\frac{1}{2}+o(1) \right ) \log (1 - 2c).
\end{eqnarray*}
This completes the proof.
\end{proof}

\section{Acknowledgments}
I am grateful to Ladislav Stacho, Toma\v{s} Masa\v{r}\'ik, and Jana Novotn\'a for helpful discussions about this topic.

\bibliographystyle{plain}
\bibliography{bib}
\end{document}